\documentclass[11pt]{amsart}
\usepackage{amsfonts,amssymb,amscd,amsmath,mathrsfs,amsthm,verbatim,fullpage,url}
\usepackage{graphicx,tikz,ifthen,rotating}
\newcommand{\rowname}{Catalan DPP}
\newcommand{\pathname}{\rowname{} path}

\newcommand\dyckpath[3]{
  \fill[lightgray!25]  (#1) rectangle +(#2,#2);
  \fill[fill=lightgray]
  (#1)
  \foreach \dir in {#3}{
    \ifnum\dir=0
    -- ++(1,0)
    \else
    -- ++(0,1)
    \fi
  } |- (#1);
  \draw[help lines] (#1) grid +(#2,#2);
  \draw[dashed] (#1) -- +(#2,#2);
  \coordinate (prev) at (#1);
  \foreach \dir in {#3}{
    \ifnum\dir=0
    \coordinate (dep) at (1,0);
    \else
    \coordinate (dep) at (0,1);
    \fi
    \draw[line width=2pt,-stealth] (prev) -- ++(dep) coordinate (prev);
  };
}

\newcommand\pathbox[3]{
  \fill[lightgray!25]  (#1) rectangle +(#2,#2-1);
  \fill[fill=lightgray]
  (#1)
  \foreach \dir in {#3}{
    \ifnum\dir=1
    -- ++(1,0)
    \else
    -- ++(0,1)
    \fi
  } |- (#1);
  \draw[help lines] (#1) grid +(#2,#2-1);
  \draw[dashed] (#1)+(0,1) -- +(#2-2,#2-1);
  \coordinate (prev) at (#1);
  \foreach \dir in {#3}{
    \ifnum\dir=1
    \coordinate (dep) at (1,0);
    \else
    \coordinate (dep) at (0,1);
    \fi
    \draw[line width=2pt,-stealth] (prev) -- ++(dep) coordinate (prev);
  };
}

\newcommand{\mypathbox}[3]{
\rotatebox[origin=c]{-90}{
\scalebox{0.5}{
\begin{tikzpicture}
\pathbox{#1}{#2}{#3}
\end{tikzpicture}
}
}
}

\newcounter{x}
\newcounter{y}
\newcounter{z}

\newcommand\xaxis{210}
\newcommand\yaxis{-30}
\newcommand\zaxis{90}

\newcommand\topside[3]{
  \fill[fill=white, draw=black,shift={(\xaxis:#1)},shift={(\yaxis:#2)},
  shift={(\zaxis:#3)}] (0,0) -- (30:1) -- (0,1) --(150:1)--(0,0);
}

\newcommand\leftside[3]{
  \fill[fill=lightgray, draw=black,shift={(\xaxis:#1)},shift={(\yaxis:#2)},
  shift={(\zaxis:#3)}] (0,0) -- (0,-1) -- (210:1) --(150:1)--(0,0);
}

\newcommand\rightside[3]{
  \fill[fill=darkgray, draw=black,shift={(\xaxis:#1)},shift={(\yaxis:#2)},
  shift={(\zaxis:#3)}] (0,0) -- (30:1) -- (-30:1) --(0,-1)--(0,0);
}

\newcommand\cube[3]{
  \topside{#1}{#2}{#3} \leftside{#1}{#2}{#3} \rightside{#1}{#2}{#3}
}

\newcommand\planepartition[1]{
 \setcounter{x}{-1}
  \foreach \a in {#1} {
    \addtocounter{x}{1}
    \setcounter{y}{-1}
    \foreach \b in \a {
      \addtocounter{y}{1}
      \setcounter{z}{-1}
      \ifthenelse{\equal{\b}{0}}
      {
      
      }
      {
      \foreach \c in {1,...,\b} {
        \addtocounter{z}{1}
        \cube{\value{x}}{\value{y}}{\value{z}}
        }
      }
    }
  }
}

\newcommand{\myPP}[1]{
	\scalebox{0.3}{
		\begin{tikzpicture}
			\planepartition{#1}
		\end{tikzpicture}
	}
}

\newcommand{\cattree}{
\begin{tikzpicture}
\tikzstyle{level 2}=[sibling distance=59mm] \tikzstyle{level 3}=[sibling distance=21mm] \tikzstyle{level 4}=[sibling distance=5mm] \node  (z){1}
child {node {2}
     child {node  (a) {2}
         child {node  (b) {2}
             child {node {2}}
             child {node {3}}
         }
         child {node  (d) {3}
             child {node {2}}
             child {node {3}}
             child {node {4}}
         }
     }
     child {node  (e) {3}
         child {node  (f) {2}
             child {node {2}}
             child {node {3}}
         }
         child {node  (g) {3}
             child {node {2}}
             child {node {3}}
             child {node {4}}
         }
         child {node  (h) {4}
             child {node {2}}
             child {node {3}}
             child {node {4}}
             child {node {5}}
         }
     }
     };
     \end{tikzpicture}
}

\newcommand{\permtree}{
\begin{tikzpicture}
\tikzstyle{level 2}=[sibling distance=63mm]
\tikzstyle{level 3}=[sibling distance=27mm]
\tikzstyle{level 4}=[sibling distance=8mm]
\node  (z){$\emptyset$}
child{node{1}
    child {node  (a) {21}
        child {node  (b) {321}
            child {node {\scriptsize 4321}}
			child {node {\scriptsize 3214}}
        }
        child {node  (d) {213}
            child {node {\scriptsize 4213}}
			child {node {\scriptsize 2143}}
			child {node {\scriptsize 2134}}
        }
    }
    child {node  (e) {12}
        child {node  (f) {312}
            child {node {\scriptsize 4312}}
			child {node {\scriptsize 3124}}
        }
        child {node  (g) {132}
            child {node {\scriptsize 4132}}
			child {node {\scriptsize 1432}}
			child {node {\scriptsize 1324}}
        }
        child {node  (h) {123}
            child {node {\scriptsize 4123}}
			child {node {\scriptsize 1423}}
			child {node {\scriptsize 1243}}
			child {node {\scriptsize 1234}}            
        }
    }};    
	\end{tikzpicture}
}

\newcommand{\pathtree}{
\begin{tabular}{r@{ }l}
\begin{tikzpicture}
\tikzstyle{level 2}=[sibling distance=59mm] \tikzstyle{level 3}=[sibling distance=21mm] \tikzstyle{level 4}=[sibling distance=5mm] \node  (z){$\emptyset$}[grow = right]
child {node {$\emptyset$}
     child {node  (a) {$\emptyset$}
         child {node  (b) {$\emptyset$}
             child {node {$\emptyset$}}
             child {node {1}}
         }
         child {node  (d) {1}
             child {node {1-11}}
             child {node {11-1}}
             child {node {11}}
         }
     }
     child {node  (e) {1}
         child {node  (f) {1-11}
             child {node {1-11-11}}
             child {node {11-1-11}}
         }
         child {node  (g) {11-1}
             child {node {1-111-1}}
             child {node {11-11-1}}
             child {node {111-1-1}}
         }
         child {node  (h) {11}
             child {node {1-111}}
             child {node {11-11}}
             child {node {111-1}}
             child {node {111}}
         }
     }
     };
     \end{tikzpicture}
\end{tabular}
}

\newcommand{\rowtree}{
\begin{tikzpicture}
\tikzstyle{level 2}=[sibling distance=59mm] \tikzstyle{level 3}=[sibling distance=21mm] \tikzstyle{level 4}=[sibling distance=5mm] \node  (z){$\emptyset$}[grow = right]
child {node {$\emptyset$}
     child {node  (a) {$\emptyset$}
         child {node  (b) {$\emptyset$}
             child {node {$\emptyset$}}
             child {node {2}}
         }
         child {node  (d) {2}
             child {node {3 2}}
             child {node {3 1}}
             child {node {3}}
         }
     }
     child {node  (e) {2}
         child {node  (f) {3 2}
             child {node {4 3 2}}
             child {node {4 2 2}}
         }
         child {node  (g) {3 1}
             child {node {4 3 1}}
             child {node {4 2 1}}
             child {node {4 1 1}}
         }
         child {node  (h) {3}
             child {node {4 3}}
             child {node {4 2}}
             child {node {4 1}}
             child {node {4}}
         }
     }
     };
     \end{tikzpicture}
}

\title{A Catalan Subset of Descending Plane Partitions}
\author{Colton Keller and Jessica Striker}
\thanks{}
\email{c.keller@ndsu.edu, jessica.striker@ndsu.edu}
\address{Department of Mathematics, North Dakota State University, PO Box 6050, Fargo, ND 58108-6050}

\newtheorem{thm}{Theorem}

\newtheorem{lemma}[thm]{Lemma}

\theoremstyle{definition}
\newtheorem{defn}[thm]{Definition}

\begin{document}

\begin{abstract}
Descending plane partitions, alternating sign matrices, and totally symmetric self-complementary plane partitions are equinumerous combinatorial sets for which no explicit bijection is known. In this paper, we isolate a subset of descending plane partitions counted by the Catalan numbers. The proof follows by constructing a generating tree on these descending plane partitions that has the same structure as the generating tree for 231-avoiding permutations. We hope this result will provide insight on the search for a bijection with alternating sign matrices and/or totally symmetric self-complementary plane partitions, since these also contain Catalan subsets.
\end{abstract}

\maketitle

\section{Introduction}
Descending plane partitions  with largest part at most $n$, alternating sign matrices with $n$ rows and $n$ columns, and totally symmetric self-complementary plane partitions inside a $2n\times 2n\times 2n$ box are each enumerated by the product formula
\begin{equation}
\label{eq:prod}
\prod_{j=0}^{n-1}\frac{(3j+1)!}{(n+j)!}.\tag{$*$}
\end{equation}
This cries out for a \emph{bijective} explanation, but it has been an outstanding problem for decades to find these missing bijections. See \cite{ANDREWS_DPP,ANDREWS_PPV,book,kuperbergASMpf,MRRANDREWSCONJ,MRRASMDPP,MRR3,RobbinsRumseyDet,ZEILASM}  for these enumerations and bijective conjectures.
Many papers have been written in an effort to reduce the problem; papers relevant to the descending plane partition portion of the problem include R.~Behrend, P.~Di Francesco, and P.~Zinn-Justin's proofs of multiply-refined enumerations of alternating sign matrices and descending plane partitions \cite{WeightedASMDPP,Behrend} and the second author's bijective map between descending plane partitions with no special parts and permutation matrices \cite{dppandpmbijection}.
Descending plane partitions may arguably be the least natural of these three sets of objects, but they are the only one  for which a statistic is known whose generating function is a $q$-analogue of (\ref{eq:prod}). This statistic is particularly nice; it is the sum of the entries (conjectured in~\cite{ANDREWS_DPP} and proved in~\cite{MRRANDREWSCONJ}). 
Descending plane partitions can be seen as non-intersecting lattice paths~\cite{lalonde2003} and also as rhombus tilings of a  hexagon with a triangular hole in the middle~\cite{KrattDPP}.

\textbf{Our main result} (Theorem~\ref{lem:catrelat}) isolates a subset of descending plane partitions of order $n$ and proves it to be enumerated by the $n$th Catalan number $C_n = \frac{1}{n+1}\binom{2n}{n}$. This result is a step toward better understanding how descending plane partitions relate to the other objects enumerated by (\ref{eq:prod}), since there are several known Catalan subsets within those sets. These include
\emph{link patterns} in \emph{fully-packed loops}, diagonals of \emph{monotone and magog triangles} (equivalently, order ideals within layers of \emph{tetrahedral posets}~\cite{STRIKERPOSET}), and endpoints of certain nests of non-intersecting  lattice paths~\cite{DiFrancesco_TSSCPP_qKZ}.
 
We hope that comparing Catalan descending plane partitions with the known Catalan subsets of alternating sign matrices and totally symmetric self-complementary plane partitions will be helpful in finding these missing bijections. We also remark that we prove Theorem~\ref{lem:catrelat} neither by finding a bijection to typical Catalan objects, such as Dyck paths or triangulations, nor by showing these objects satisfy the Catalan recurrence in the usual way. Rather, we use the generating tree approach of J.~West~\cite{CatTrees}, which we have found to be both lovely and useful and should be more widely known.

The paper is organized as follows. In Section~\ref{sec:def}, we first define descending plane partitions. We then define the other objects enumerated by (\ref{eq:prod}) and discuss their known Catalan subsets. In Section~\ref{sec:catdpp}, we identify our Catalan subset of descending plane partitions  and state our main result, Theorem~\ref{lem:catrelat}, on the enumeration of this subset. In Section~\ref{sec:gentree}, we give background on the Catalan generating tree and 231-avoiding permutations. In Section~\ref{sec:lemmaproof}, we prove Theorem~\ref{lem:catrelat}. Finally, in Section~\ref{sec:disc}, we discuss some implications of our main theorem.

\section{Definitions and background}
In this section, we first state definitions related to descending plane partitions. We then define alternating sign matrices and totally symmetric self-complementary plane partitions and briefly discuss their Catalan subsets.

\label{sec:def}
\begin{defn}
A \textbf{plane partition} is a finite subset, $P$, of positive integer lattice points, $\{(i,j,k)\subset \mathbb{N}^3\}$, such that if $(r,s,t)\in P$ and $(i,j,k)$ satisfies $1 \leq i \leq r$, $1 \leq j \leq s$, and $1 \leq k \leq t$, then $(i,j,k)\in P$ as well.
\end{defn}
A plane partition may be visualized as a stack of unit cubes pushed into a corner as shown in Figure~\ref{fig:planepartition} below. A plane partition may also be represented as a left and top justified array of integers where each row represents one layer of the plane partition and each integer represents the number of unit cubes stacked in that column of the layer. For example, Figure \ref{fig:planepartition}, left, shows a plane partition with its corresponding integer array. Note the rows and columns of an integer array corresponding to a plane partition must be weakly decreasing.

\begin{defn}
A \textbf{shifted plane partition} is an array of integers weakly decreasing in rows and columns such that the beginning of each row is shifted to the right by one position relative to the previous row. That is, it is an integer array as in Figure \ref{fig:jesslabeling} where each row and column is weakly decreasing. A \textbf{strict shifted plane partition} is a shifted plane partition with the additional restriction that each column must be strictly decreasing.
\end{defn}

Figure \ref{fig:planepartition}, center, shows a strict shifted plane partition.

\begin{figure}[htbp]
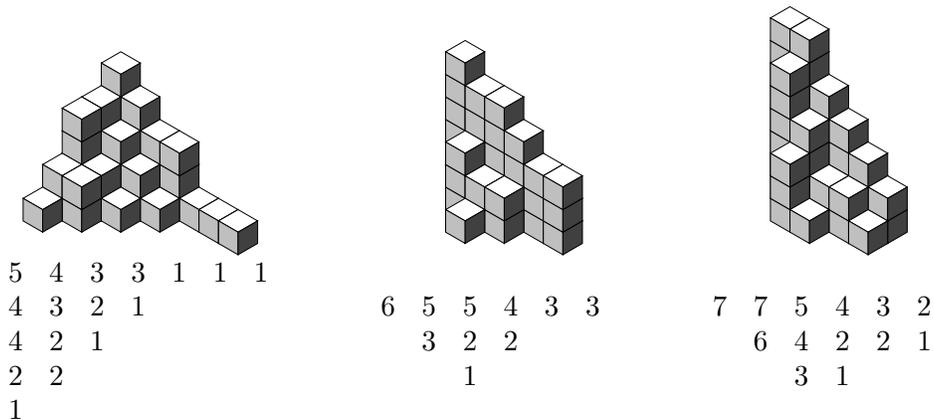

\myPP{{5,4,3,3,1,1,1},{4,3,2,1},{4,2,1},{2,2},{1}} \hspace{.75in} \myPP{{6,5,5,4,3,3},{0,3,2,2},{0,0,1}} \hspace{.75in} \myPP{{7,7,5,4,3,2},{0,6,4,2,2,1},{0,0,3,1}}

$\begin{array}{ccccccc}
5&4&3&3&1&1&1\\
4&3&2&1\\
4&2&1\\
2&2\\
1
\end{array}$ \hspace{.35in}
$\begin{array}{cccccc}6&5&5&4&3&3\\ &3&2&2\\ & &1\end{array}$ \hspace{.35in}
$\begin{array}{cccccc}7&7&5&4&3&2\\ &6&4&2&2&1\\ & &3&1\end{array}$
\caption{\label{fig:planepartition}Left: A plane partition, visualized as unit cubes in a corner, with its corresponding integer array representation; Center: A strict shifted plane partition; Right: A descending plane partition.}
\end{figure}

Our main objects of study are the following.
\begin{defn}
\label{def:dpp}
A \textbf{descending plane partition} (DPP) is a strict shifted plane partition with the additional restrictions that:
\begin{enumerate}
\item the number of parts in each row is less than the greatest part in that row, and
\item the number of parts in each row is greater than or equal to the largest part in the next row.
\end{enumerate}
Using the notation in Figure \ref{fig:jesslabeling}, the above conditions are equivalent to:
\begin{enumerate}
\item $\lambda_{k} - k + 1<a_{k,k}$, and
\item $\lambda_{k} - k + 1\geq{a_{k+1,k+1}}$.
\end{enumerate}
A descending plane partition is \textbf{of order n} if its largest part is less than or equal to $n$.
\end{defn}

\begin{figure}[htbp]
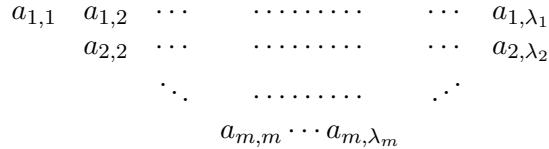

\[\begin{array}{cccccccc}
a_{1,1}&a_{1,2}&\cdots&\cdots\cdots\cdots&\cdots&a_{1,\lambda_{1}}\\
 &a_{2,2}&\cdots&\cdots\cdots\cdots&\cdots&a_{2,\lambda_{2}}\\
 & &\ddots & \cdots\cdots\cdots & \text{\reflectbox{$\ddots$}}\\
 & & &a_{m,m} \cdots a_{m,\lambda_{m}}
\end{array}\]
\caption{\label{fig:jesslabeling}Labeling of positions in shifted, strict shifted, and descending plane partitions.}
\end{figure}

An example of a descending plane partition is given in Figure \ref{fig:planepartition}, right.
See Figure~\ref{fig:dppsorder3} for the seven descending plane partitions of order 3.
\begin{figure}[htbp]
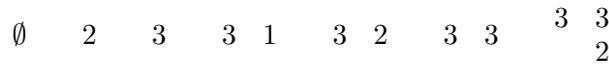

$\begin{array}{c}\emptyset\end{array}$  \ \
$\begin{array}{c}2\end{array}$  \ \
$\begin{array}{c}3\end{array}$  \ \
$\begin{array}{cc}3&1\end{array}$ \ \  
$\begin{array}{cc}3&2\end{array}$  \ \
$\begin{array}{cc}3&3\end{array}$  \ \
$\begin{array}{cc}3&3\\ &2\end{array}$  
\caption{\label{fig:dppsorder3}The seven descending plane partitions of order 3.}
\end{figure}

In the next section, we define our Catalan subset of descending plane partitions, which we prove in Theorem~\ref{lem:catrelat} to be counted by the Catalan numbers. We finish this section with definitions of the other objects enumerated by (\ref{eq:prod}) and a short discussion of their Catalan subsets.

\begin{defn}
An \textbf{alternating sign matrix} is a square matrix with entries in $\{0,1,-1\}$, row and column sums equal to 1, and the additional restriction that the nonzero entries alternate in sign across each row and column. 
\end{defn}

\begin{figure}[htbp]
$\left(\begin{array}{rrrrrr}0&0&0&1&0&0\\0&1&0&\mbox{--}1&0&1\\1&\mbox{--}1&0&0&1&0\\0&1&0&0&0&0\\0&0&1&0&0&0\\0&0&0&1&0&0\end{array}\right) \  
\begin{array}{ccccccccccc}
& & & & & 4 & & & & & \\
& & & & 2 & & 6 & & & & \\
& & & 1 & & 5 & & 6 & & & \\
& & 1 & & 2 & & 5 & & 6 & & \\
& 1 & & 2 & & 3 & & 5 & & 6 & \\
1 & & 2 & & 3 & & 4 & & 5 & & 6 
\end{array} \ 
\vcenter{\hbox{\includegraphics[scale=.25]{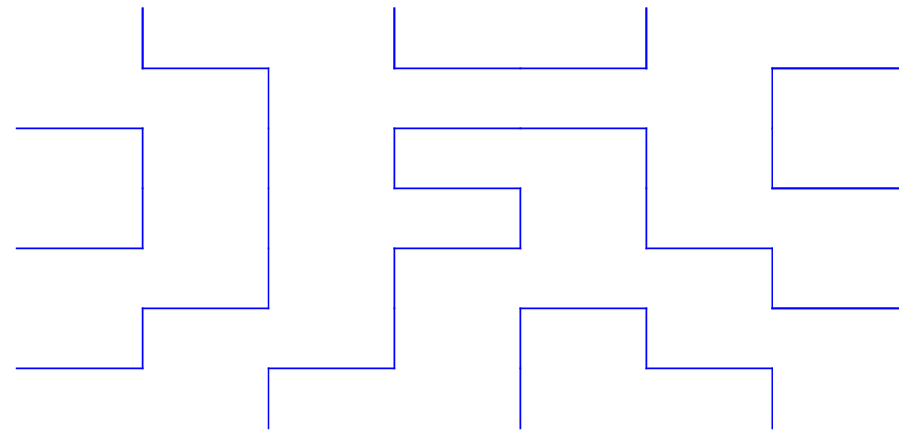}}}$
\caption{\label{fig:asmexample}Left: An alternating sign matrix; Center: Its corresponding monotone triangle; Right: Its corresponding fully-packed loop configuration.}
\end{figure}

An example alternating sign matrix can be seen in Figure \ref{fig:asmexample}, left. 
Note any permutation matrix is also an alternating sign matrix. The seven $3\times3$ alternating sign matrices are given in Figure \ref{fig:sevenasms}.

\begin{figure}[htbp]
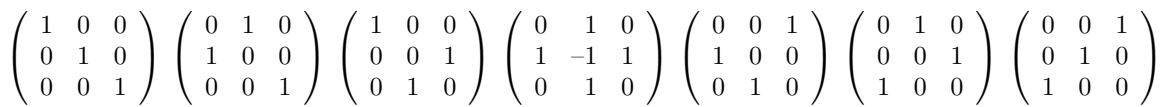

\scalebox{.9}{
$\left(\begin{array}{rrr}1&0&0\\0&1&0\\0&0&1\end{array}\right)$
$\left(\begin{array}{rrr}0&1&0\\1&0&0\\0&0&1\end{array}\right)$
$\left(\begin{array}{rrr}1&0&0\\0&0&1\\0&1&0\end{array}\right)$
$\left(\begin{array}{rrr}0&1&0\\1&\mbox{--}1&1\\0&1&0\end{array}\right)$
$\left(\begin{array}{rrr}0&0&1\\1&0&0\\0&1&0\end{array}\right)$
$\left(\begin{array}{rrr}0&1&0\\0&0&1\\1&0&0\end{array}\right)$
$\left(\begin{array}{rrr}0&0&1\\0&1&0\\1&0&0\end{array}\right)$}
\caption{\label{fig:sevenasms}The seven $3\times3$ alternating sign matrices.}
\end{figure}

Alternating sign matrices are in bijection with monotone triangles and fully-packed loop configurations. See Figure~\ref{fig:asmexample} for an example; see also~\cite{ProppManyFaces}. Each length $n$ northwest to southeast diagonal of a monotone triangle is a weakly increasing sequence of integers $x_1,x_2,\ldots,x_n$ such that $i\leq x_i\leq n$; these are counted by the $n$th Catalan number $C_n$.
Also, each fully-packed loop configuration of order $n$ has an associated link pattern, which is the noncrossing matching on its $2n$ external vertices; these are also counted by $C_n$.

\begin{defn}
A plane partition is \textbf{totally symmetric} if whenever $(i,j,k)\in\pi$, then all six permutations of $(i,j,k)$ are also in $\pi$. A plane partition is \textbf{self-complementary} inside its $a\times b\times c$ bounding box if it is equal to its complement in the box, that is, the
collection of empty cubes in the box is of the same shape as the
collection of cubes in the plane partition itself.
\end{defn}

\begin{figure}[htbp]
$\vcenter{\hbox{\includegraphics[scale=.42]{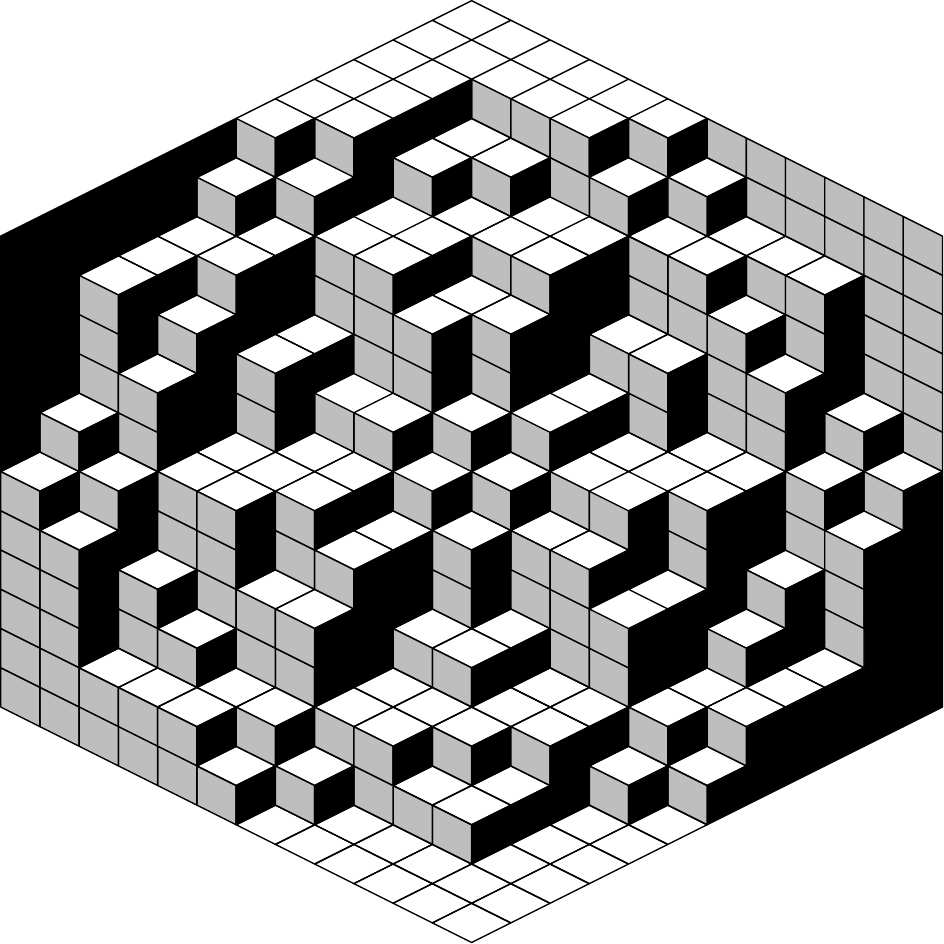}}} \  
\begin{array}{ccccccccccc}
& & & & & 6 & & & & & \\
& & & & 4 & & 5 & & & & \\
& & & 3 & & 4 & & 5 & & & \\
& & 2 & & 3 & & 5 & & 6 & & \\
& 1 & & 2 & & 3 & & 4 & & 6 & \\
1 & & 2 & & 3 & & 4 & & 5 & & 6 
\end{array} \ \
\vcenter{\hbox{\includegraphics[scale=.6]{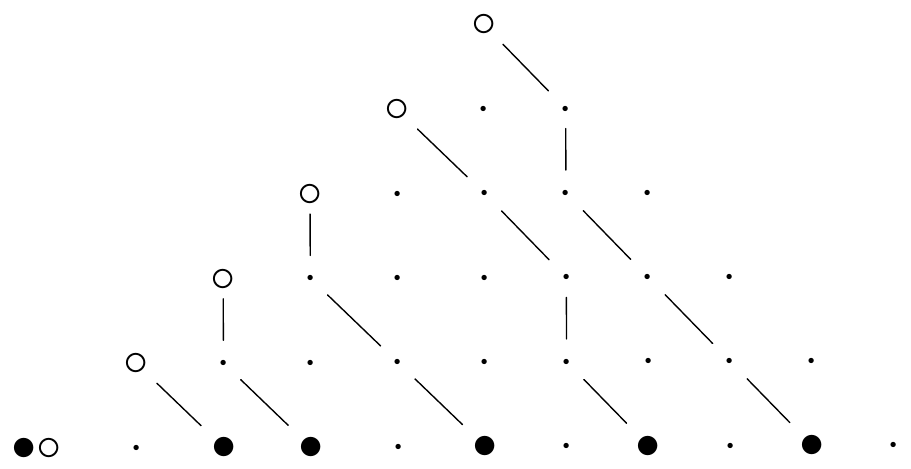}}}$
\caption{\label{fig:tsscppex}Left: A totally symmetric self-complementary plane partition inside a $12\times 12\times 12$ box; Center: Its corresponding magog triangle; Right: Its corresponding nest of non-intersecting lattice paths.}
\end{figure}

An example totally symmetric self-complementary plane partition can be seen in Figure \ref{fig:tsscppex}, left. 
The seven totally symmetric self-complementary plane partitions inside a $6\times 6\times 6$ box are given in Figure~\ref{ex:tikztsscpp}.

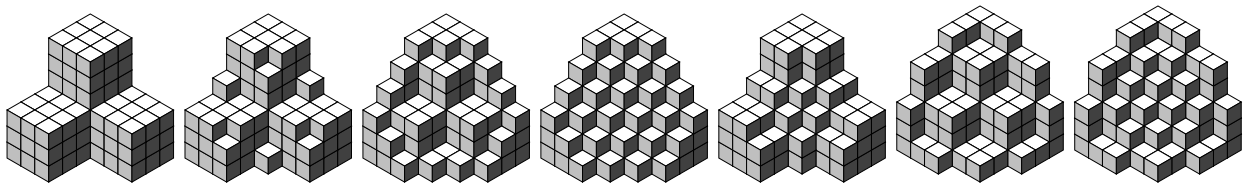
\begin{figure}[htbp]
\begin{center}
$\begin{gathered}\scalebox{0.213}{
$\begin{tikzpicture}
\planepartition{{6,6,6,3,3,3},{6,6,6,3,3,3},{6,6,6,3,3,3},{3,3,3},{3,3,3},{3,3,3}}
\end{tikzpicture}
\hspace{.7cm}
\begin{tikzpicture}
\planepartition{{6,6,6,4,3,3},{6,6,6,3,3,3},{6,6,5,3,3,2},{4,3,3,1},{3,3,3},{3,3,2}}
\end{tikzpicture}
\hspace{.7cm}
\begin{tikzpicture}
\planepartition{{6,6,6,5,4,3},{6,6,5,3,3,2},{6,5,5,3,3,1},{5,3,3,1,1},{4,3,3,1},{3,2,1}}
\end{tikzpicture}
\hspace{.7cm}
\begin{tikzpicture}
\planepartition{{6,6,6,5,4,3},{6,6,5,4,3,2},{6,5,4,3,2,1},{5,4,3,2,1},{4,3,2,1},{3,2,1}}
\end{tikzpicture}
\hspace{.7cm}
\begin{tikzpicture}
\planepartition{{6,6,6,4,3,3},{6,6,6,4,3,3},{6,6,4,3,2,2},{4,4,3,2},{3,3,2},{3,3,2}}
\end{tikzpicture}
\hspace{.7cm}
\begin{tikzpicture}
\planepartition{{6,6,6,5,5,3},{6,5,5,3,3,1},{6,5,5,3,3,1},{5,3,3,1,1},{5,3,3,1,1},{3,1,1}}
\end{tikzpicture}
\hspace{.7cm}
\begin{tikzpicture}
\planepartition{{6,6,6,5,5,3},{6,5,5,4,3,1},{6,5,4,3,2,1},{5,4,3,2,1},{5,3,2,1,1},{3,1,1}}
\end{tikzpicture}$
}\end{gathered}$
\end{center}
\caption{The seven totally symmetric self-complementary plane partitions inside a $6\times 6\times 6$ box.}
\label{ex:tikztsscpp}
\end{figure}

Totally symmetric self-complementary plane partitions inside a $2n\times 2n\times 2n$ box are in bijection with magog triangles and certain non-intersecting lattice paths. See Figure~\ref{fig:tsscppex} for an example; see also~\cite{PermTSSCPP}. Each length $n$ northwest to southeast diagonal of a magog triangle is a sequence of integers $x_1,x_2,\ldots,x_n$ such that $i\leq x_i\leq n$ and $x_{i+1}\leq x_{i}+1$; these are counted by $C_n$. The number of possible sequences of dots and blank spots across the bottom row of the non-intersecting lattice path representation is also counted by $C_n$, via a bijection to Dyck paths. 

\section{Catalan descending plane partitions}
\label{sec:catdpp}
In this section, we identify a subset of descending plane partitions that can be counted by the Catalan numbers and state our main result, Theorem~\ref{lem:catrelat}. We will prove this theorem in Section~\ref{sec:lemmaproof} using the Catalan generating tree, discussed in Section~\ref{sec:gentree}.

\begin{defn}
A \textbf{Catalan descending plane partition (\rowname)} is a one-row descending plane partition $a_{1,1} \ a_{1,2} \  \cdots \ a_{1,{\lambda_1}}$  such that each entry satisfies $a_{1,j} \leq a_{1,1} - j + 1$. 
\end{defn}

The Catalan descending plane partitions of order 4 are shown in Figure \ref{fig:specialdpprows}.

\begin{figure}[htbp]
$\emptyset$, 2, 3, 3 1, 3 2, 4, 4 1, 4 2, 4 3, 4 1 1, 4 2 1, 4 3 1, 4 2 2, 4 3 2
\caption{\label{fig:specialdpprows}The 14 Catalan descending plane partitions of order 4.}
\end{figure}

Our main theorem justifies the naming of these objects.
\begin{thm}
\label{lem:catrelat}
Catalan descending plane partitions of order $n$ are counted by the $n$th Catalan number \[C_n= \displaystyle\frac{1}{n + 1}\displaystyle\binom{2n}{n}.\]
\end{thm}

In our proof of Theorem~\ref{lem:catrelat}, we will represent Catalan descending plane partitions as paths. 

\begin{defn}
A \textbf{\pathname} of order $n$ is a sequence of entries in $\{-1,1\}$ with at most $n-1$ ones such that the left-to-right partial sums are nonnegative and the total sum is greater than zero (if the sequence is non-empty).
\end{defn}

A list of the first 14 \pathname{s} is given in Figure \ref{fig:words}. 

\begin{figure}[htbp]
$\emptyset$, 1, 11, 11-1, 1-11, 111, 111-1, 11-11, 1-111, 111-1-1, 11-11-1, 1-111-1, 11-1-11, 1-11-11 
\caption{\label{fig:words}The 14 Catalan DPP paths of order 4, in order corresponding to Figure~\ref{fig:specialdpprows} via the bijection of Lemma~\ref{lem:bijDPPtoDPPpath}.}
\end{figure}

\pathname{s} are in natural bijection with \rowname{s}.

\begin{lemma}
\label{lem:bijDPPtoDPPpath}
\rowname{s} of order $n$ are in bijection with \pathname{s} of order $n$.
\end{lemma}

\begin{proof}
The empty \rowname{} $\emptyset$ corresponds to the empty \pathname{} $\emptyset$.
For nonempty \rowname{s}, following the convention in the definitions of the previous section, we are interested in the southwest wall projected onto a grid; see Figure~\ref{fig:rowtopath}. This side view of the \rowname{} nicely corresponds with the numeric representation, in that each entry $a_{1,j}$ of the \rowname{} is represented as a set of $a_{1,j}$ boxes aligned with the bottom-left of the graph in sequential order. 

Consider the path formed by tracing the upper edge of the projected \rowname{} from upper left to lower right. Representing this path as a sequence of $1$'s and $-1$'s, where each $-1$ corresponds to a move to the right and each $1$ corresponds to a move down, we attain a word beginning with a $-1$ and ending with a $1$. Note that there cannot be a move down without first a move to the right to define the top of the first set of boxes, and there must be a final descent to finish the \rowname{} on the last entry. Ignoring these first and last entries, we have a valid \pathname. Each $-1$ in a \pathname{} corresponds to an increase in the $j$ component of each $a_{1,j}$ of the \rowname, so the requirement of partial sums greater than or equal to zero corresponds to the \rowname{} requirement that $a_{1,j}$ be less than or equal to $a_{1,1} - j + 1$. The requirement that the sum of each \pathname{} be greater than zero ensures that the \rowname{} represented by the path is a valid DPP with the length of each row less than the largest part in the row. This process is clearly invertible, and is thus a bijection.
\end{proof} 

See Figure~\ref{fig:rowtopath} for an example of this bijection.

\begin{figure}[htbp]
$432$ \ {$\Leftrightarrow$} \ \raisebox{-.8cm}{\myPP{{4,3,2}}} \ {$\Leftrightarrow$} \ \raisebox{-.6cm}{\mypathbox{0,0}{4}{-1,1,-1,1,-1,1,1}} \ {$\Leftrightarrow$} \ \textcolor{red}{-1}1-11-11\textcolor{red}{1} \ {$\Leftrightarrow$} \ 1-11-11
\caption{\label{fig:rowtopath}An example of the bijection between a \rowname{} and its \pathname{}.}
\end{figure}

\section{231-avoiding permutations and the Catalan generating tree}
\label{sec:gentree}
To prove Theorem \ref{lem:catrelat}, we will use the Catalan generating tree to show that \pathname{s} are in bijection with 231-avoiding permutations, which are known to be counted by $C_n$.

\begin{defn}
A \textbf{$\mathbf{231}$-avoiding permutation} is a permutation 
$\sigma = \sigma_1 \sigma_2 \cdots \sigma_n$ with no triple of indices $\{i,j,k\}$ such that $i<j<k$ and $\sigma_k < \sigma_i < \sigma_j$.
\end{defn}

It is well-known that permutations of $n$ avoiding any permutation $\pi\in\mathfrak{S}_3$ are counted by $C_n$; for example, see \cite{stanley2015catalan}. 
We give below a proof for the case $\pi=231$.

\begin{lemma}
231-avoiding permutations in $\mathfrak{S}_n$ are counted by $C_n$. 
\end{lemma}

\begin{proof}
Let $X_n$ denote the number of 231-avoiding permutations of $\mathfrak{S}_n$. It is clear that if $\sigma\in\mathfrak{S}_n$ is a 231-avoiding permutation, then any $\sigma_i \cdots \sigma_j$ for $1\leq{i}\leq{j}\leq{n}$, is also a 231-avoiding permutation (when each $\sigma_k$ is standardized to fall between $1$ and $j-i+1$ while maintaining relative order). Suppose $\sigma_p=n$. We may break $\sigma$ at position $p$ into permutations $\sigma_1 \cdots \sigma_{p-1}$ and $\sigma_{p+1} \cdots \sigma_{n}$ (allowing the empty permutation), which are still 231-avoiding when standardized.
This gives the Catalan recurrence relation $X_{n} = \displaystyle\sum_{p=1}^{n}X_{p-1} X_{n-p}, \ X_0=1$, 
since we obtain two new 231-avoiding permutations of length $p-1$ and $n-p$. 
\end{proof}

In \cite{CatTrees}, J.~West describes a generating tree whose nodes are in bijection with 231-avoiding permutations. The structure of these trees can be described as follows:

\begin{defn}
To construct the \textbf{Catalan generating tree}, begin with the first node at level $n = 0$. If a node is at level $n$, we say its children are at level $n+1$, and its parent is at level $n-1$.  To determine how many children each node should have, if a node is in the $p$th position from the left among its siblings, it will have $p+1$ children. The convention for the node at level $n=0$ is that it has one child. 
\end{defn}

\begin{lemma}[\cite{CatTrees}]
\label{lem:cattree}
The number of nodes at level $n$ in the Catalan generating tree is $C_n$. 
\end{lemma}

\begin{proof}
The nodes of the Catalan generating tree will be populated with 231-avoiding permutations following West's procedure given in \cite{CatTrees}. First, the root (level $n=0$) will be labeled with $\emptyset$. The nodes at level $n$ for $n\geq 1$ will be 231-avoiding permutations of length $n$. The children of each node will be found by inserting $n$ into the permutation at all valid locations that maintain the 231-avoiding property, starting from left to right. By doing this, we ensure that each node generated is a valid 231-avoiding permutation. We say that $\sigma_i$ is a left-to-right maximum of $\sigma$ if $\sigma_i > \sigma_k$ for all $k<i$. The valid locations for insertion are immediately to the left of a left-to-right maximum or at the end of the permutation. This will generate every 231-avoiding permutation for a given $n$.  Inserting the new $n$ into the permutation affects the descendants predictably. If a node has $k$ left-to-right maxima, it will have $k+1$ children (the additional child comes from adding the $n$ to the end of the permutation) that each have $1,2,3,\cdots,k+1$ left-to-right maxima, respectively, and so $2,3,4,\cdots,k+2$ children, respectively. This pattern follows exactly the structure of Catalan generating tree as defined above, and so the Catalan tree nodes at level $n$ are counted by $C_n$, as desired.
\end{proof}

The start of the Catalan generating tree is shown in Figure \ref{fig:cattree}, and the correspondence with 231-avoiding permutations  can be seen in Figure \ref{fig:permtree}.

\begin{figure}[htbp]
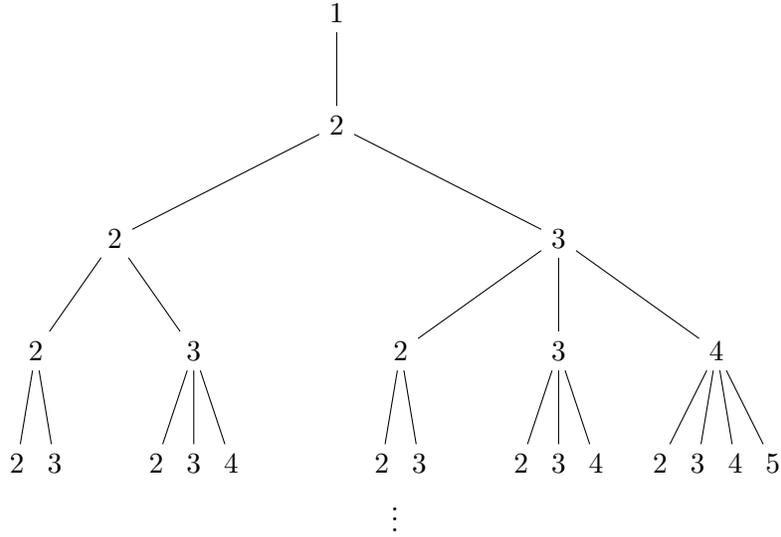

\begin{center}
\cattree
\end{center}

\text{\vdots}

\caption{\label{fig:cattree}The first five levels of the Catalan generating tree described in \cite{CatTrees}. Each node is labeled with the number of children it has.}
\end{figure}

\begin{figure}[htbp]
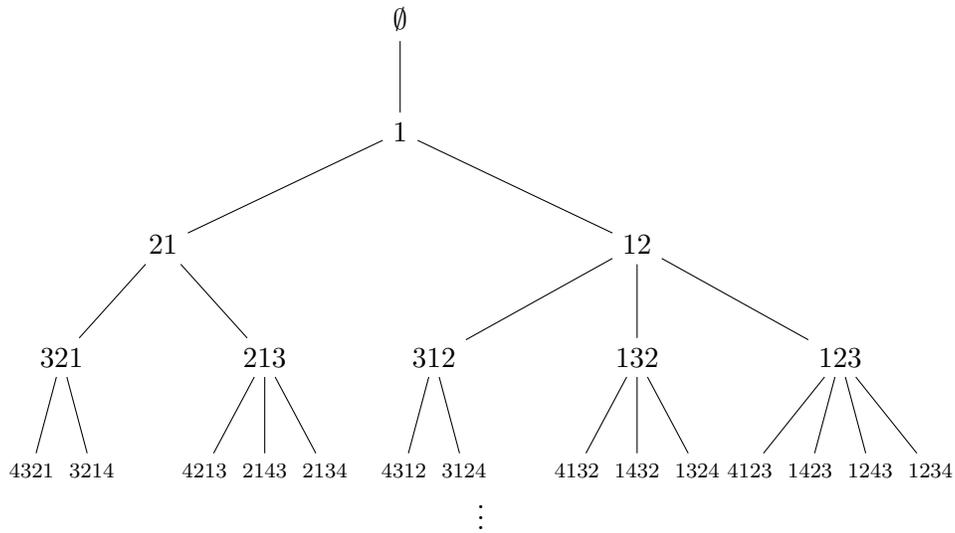

\permtree

\text{\vdots}
\caption{\label{fig:permtree}The first five levels of the Catalan generating tree with nodes labeled by the 231-avoiding permutations.}
\end{figure}

\section{Proof of Theorem \ref{lem:catrelat}}\label{sec:lemmaproof}

We will obtain a bijection between \rowname{s} and 231-avoiding permutations through \pathname{s} by constructing a generating tree with \pathname{s} as the nodes. The $n$th level of this tree will consist of all the \pathname{s} of order $n$; see Figure \ref{fig:pathtree}.

\begin{defn} To construct the \textbf{Catalan DPP path generating tree}, begin with the topmost node, which will be $\emptyset$. To determine the children of a given Catalan DPP path $P$, we have the following three cases. 

\begin{itemize}
\item
\textbf{Case $P=\emptyset$:} We must check to see if $P$ has a parent. If $P$ has a parent, its children will be $\emptyset$ and $1$. If $P$ does not have a parent, it will have one child, $\emptyset$. 
\item
\textbf{Case $P\neq\emptyset$ and $P$ contains a -1:} The rightmost of $P$'s children will be $P$ with an additional pair 1 -1 inserted directly before the first occurrence of -1 in $P$. This produces a valid \pathname{}, since the total sum of the child \pathname{} will remain the same or increase, and adding a sequence of 1-1 will not produce a partial sum less than the previous partial sum. Clearly, inserting a 1 anywhere in the \pathname{} will produce another valid path.  After this first child is generated, the remaining siblings from right to left can be found by shifting the position of the leftmost -1 one position further left, unless doing so would  cause the path to no longer be a valid \pathname{} (in other words, unless the -1 would be shifted to the left of the first 1). 
\item
\textbf{Case $P\neq\emptyset$ and $P$ contains no -1:} The rightmost of $P$'s children will be $P$ with an additional 1 inserted at the end. The next sibling \pathname{} can be found by appending a -1 to the end of the rightmost child \pathname. The remaining siblings can be found by shifting this -1 left as far as possible, as before. 
\end{itemize}
\end{defn}

To prove Theorem \ref{lem:catrelat}, we show that the Catalan DPP path generating tree is in bijection with the Catalan generating tree, discussed in the previous section. 

\begin{proof}[Proof of Theorem \ref{lem:catrelat}]
It is clear that the first three levels of the \pathname{} generating tree are in bijection with the Catalan generating tree, and each leftmost node except the first will always have two children, as the leftmost node is always $\emptyset$. To see how the position of each node relative to its siblings determines how many children it will have, we will consider two sibling nodes, $A$ and $B$, immediately next to each other with $A$ being on the left and $B$ being on the right. To see how $B$ will always have one more child than $A$, there are three cases to consider: 
\begin{itemize}
\item
\textbf{Case $A=\emptyset$ and $B=1$:} It is clear that $A$'s children will be 1 and $\emptyset$. $B$'s children from right to left will be 11, 11-1, and 1-11. Thus $B$ has one more child than $A$, as desired.
\item
\textbf{Case $A,B\neq\emptyset$ and $B$ contains no -1:} It is clear from the way that the tree is constructed that $B$ will have the same number of 1's as $A$. The first of $B$'s children from right to left will be $B$ with an additional 1 appended to it. The second child of $B$ will be the same as the first with an additional -1 appended to it. This second child of $B$ will be the same as the first child of $A$ after removing the last -1 from the child of $A$. Clearly, there will be the same number of remaining siblings generated for each of the cousin nodes, as there are the same positions available for the leftmost -1 to be shifted to. Thus, $B$ has one more child than $A$. 
\item
\textbf{Case both $A$ and $B$ contain a -1:} The leftmost -1 of $A$ will be one position further left than the leftmost -1 of $B$, due to the way $A$ is generated from $B$. Thus, after inserting the 1 -1 to generate the first children of both $A$ and $B$, there will be one additional sibling generated for the first child of $B$ compared to the first child of $A$, so $B$ has one more child than $A$. 
\end{itemize}
In all of these cases, $B$ has one more child than $A$ due to its position relative to $A$, just as each node in the Catalan tree has one more child than its closest sibling to the left. This ensures the structure of the \pathname{}  tree generated is the same as that of the Catalan generating tree. 

To show that this mapping is one-to-one, we will show that each node can only have one unique parent, and so every node must be unique on that level of the \pathname{} generating tree. For any node $C$ in the tree, it is clear that it is distinct from each of its siblings due to the way the tree is constructed. To find the parent of $C$, we must perform the process of constructing the tree in reverse. If $C$ is $\emptyset$, clearly its parent is $\emptyset$. Otherwise if $C$ contains no -1's, its parent is $C$ with one of its 1's removed. There are clearly no duplicate parents possible in this case. If $C$ contains a -1, its parent can be found by removing the leftmost 1 and leftmost -1 from $C$. It is clear that any other node than this could not generate $C$ as a child, as each 1 and -1 removed from $C$ were to the left of any other -1's in $C$. Also, the parent of $C$ found with this method must generate $C$. Thus, each node in each level could only have one unique parent and could not have appeared anywhere else in that level. 

To show that every \pathname{} is contained in the \pathname{} generating tree, suppose there exists some Catalan DPP path, $D$, of order $n$, that is not present in the $n$th level of the \pathname{} tree. It is clear that $D$ cannot be $\emptyset$ since $\emptyset$ is present at each level of the \pathname{} tree. Therefore, $D$ contains at least one 1 and less than or equal to $n$ 1's and could contain any legal number of -1's. It is now clear that the method used to find the hypothetical parent of $D$ as above can be validly applied to $D$. Applying this method will always produce a new node that can have this method applied to it, or $\emptyset$ as the parent of $D$. If $D$ has $d$ 1's, it has at most $(d-1)$ -1's. Also, applying the parent-finding method recursively always yields a new node with one less 1 and one less -1 or zero -1's. Applying this parent-finding method $d$ times, we will have found $\emptyset$ as an ancestor of $D$. But if $\emptyset$ is an ancestor of $D$, this would clearly have generated a path to $D$, therefore $D$ is in the \pathname{} generating tree. 

Thus, every \pathname{} is present in its appropriate level of the \pathname{} generating tree, and there are no duplicates at any specific level. This mapping is one-to-one and onto, so therefore is a bijection, so by Lemma~\ref{lem:cattree}, \pathname{s} of order $n$ are counted by $C_{n}$.
Applying Lemma~\ref{lem:bijDPPtoDPPpath} that \pathname{s} are in bijection with \rowname{s}, we see that \rowname{s} of order $n$ are also counted by $C_n$. 
\end{proof}

\begin{figure}[htbp]
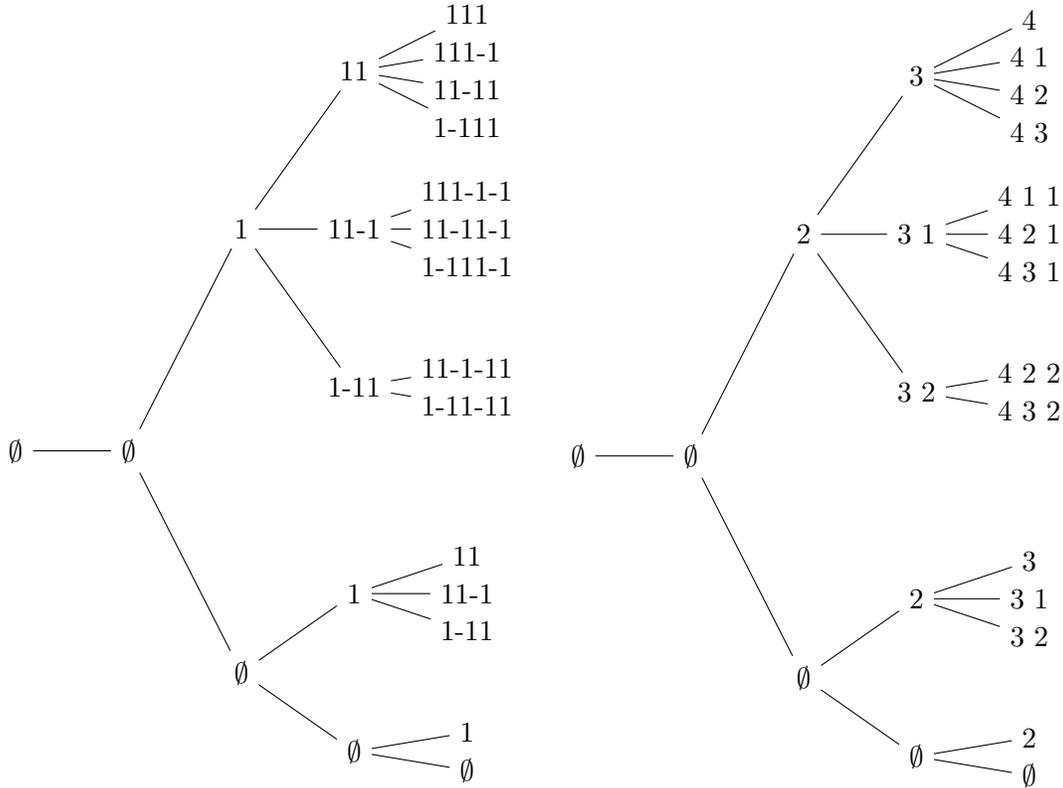

\begin{center}
$\vcenter{\hbox{\pathtree}} \ \vcenter{\hbox{\rowtree}}$
\end{center}
\caption{\label{fig:pathtree}Left: The first five levels of the \pathname{} generating tree, rotated for easier spacing; Right: The same tree with each node replaced by the corresponding \rowname. The same structure can be seen as in Figure \ref{fig:cattree}.}
\end{figure}

\section{Discussion}
\label{sec:disc}
Though our proof of Theorem~\ref{lem:catrelat} relied on expressing Catalan DPPs in path form, we can now see the Catalan generating tree structure directly on Catalan DPPs. Given a Catalan DPP $a_{1,1} \ a_{1,2} \ \cdots \ a_{1,\lambda_1}$, construct its children by incrementing $a_{1,1}$ and inserting a new number between the first two numbers in all possible ways. For example, the Catalan DPP 6 4 2 2 has children, from right to left, 7 4 4 2 2, 7 5 4 2 2, and 7 6 4 2 2. A general Catalan DPP $a_{1,1} \ a_{1,2} \ \cdots \ a_{1,\lambda_1}$ will have exactly $a_{1,1}+1-a_{1,2}$ children, which will all be of the form $(a_{1,1}+1) \ x \ a_{1,2} \ \cdots \ a_{1,\lambda_1}$ for all integers $x$ such that $a_{1,1}+1\geq x\geq a_{1,2}$. See Figure~\ref{fig:pathtree}.

We hope that comparing the Catalan descending plane partitions presented in this paper with the known Catalan subsets of alternating sign matrices and totally symmetric self-complementary plane partitions (briefly discussed in Section \ref{sec:def}) will be helpful in finding the missing explicit bijections among these three sets of objects enumerated by (\ref{eq:prod}). One challenge is that, in contrast to the Catalan objects on the other two sets, it is not clear how to associate a Catalan DPP to each non-Catalan descending plane partition in a natural way. One would hope to map a Catalan DPP to a each DPP such that the distribution over all descending plane partitions of order $n$ would match one of the Catalan distributions on alternating sign matrices or totally symmetric self-complementary plane partitions. This has not been achieved, but there is potential that it yet may. 

\section*{Acknowledgments}
The authors thank Kevin Dilks for help in coding both \url{SageMath} and \LaTeX~to produce some of the figures of this paper. We also thank Mi Huynh for directing our attention toward the Catalan generating tree in \cite{stanley2015catalan}. Striker was partially supported by  National Security Agency Grant H98230-15-1-0041.

\bibliographystyle{plain}
\bibliography{examplebib}
\end{document}